\documentclass[a4paper,11pt]{amsart}
\usepackage[left=2.7cm,right=2.7cm,top=3.5cm,bottom=3cm]{geometry}

\usepackage{amsthm,amssymb,amsmath,amsfonts,mathrsfs,amscd}
\usepackage[latin1]{inputenc}
\usepackage[all]{xy}
\usepackage{latexsym}
\usepackage{longtable}
\usepackage{calligra}

\theoremstyle{plain}
\newtheorem{theorem}{Theorem}[section]
\newtheorem{proposition}[theorem]{Proposition}
\newtheorem{lemma}[theorem]{Lemma}
\newtheorem{corollary}[theorem]{Corollary}

\theoremstyle{definition}

\newtheorem{definition}[theorem]{Definition}
\newtheorem{assumption}[theorem]{Assumption}

\theoremstyle{remark}
\newtheorem{remark}[theorem]{Remark}

\newcommand{\Q}{\mathbb Q}

\newcommand{\Z}{\mathbb Z}

\DeclareMathOperator{\Frob}{Frob}

\DeclareMathOperator{\Hom}{Hom}
\DeclareMathOperator{\Div}{Div}
\DeclareMathOperator{\Gal}{Gal}

\DeclareMathOperator{\Sel}{Sel}
\DeclareMathOperator{\M}{M}

\DeclareMathOperator{\Char}{Char}

\newcommand{\val}{\mathrm{val}}
\newcommand{\res}{\mathrm{res}}

\newcommand{\ord}{\mathrm{ord}}

\newcommand{\cont}{\mathrm{cont}}

\newcommand{\T}{\mathbf{T}}
\newcommand{\A}{\mathbf{A}}

\newfont{\cyr}{wncyr10 scaled 1100}
\newfont{\cyrr}{wncyr9 scaled 1000}

\usepackage[usenames]{color}
\definecolor{Indigo}{rgb}{0.2,0.1,0.7}
\definecolor{Violet}{rgb}{0.5,0.1,0.7}
\definecolor{White}{rgb}{1,1,1}
\definecolor{Green}{rgb}{0.1,0.9,0.2}

\newcommand{\mat}[4]{\left(\begin{array}{cc}#1&#2\\#3&#4\end{array}\right)}

\newcommand{\dirlim}{\mathop{\varinjlim}\limits}
\newcommand{\invlim}{\mathop{\varprojlim}\limits}

\newcommand{\p}{\mathfrak p}

\newcommand{\q}{\mathfrak q}

\newcommand{\pwseries}[1]{[[#1]]}

\newcommand{\F}{\mathbf f}
\newcommand{\I}{\mathcal R}

\def\k{\kappa}

\include{thebibliography}
\begin{document}

\title[AMC and Rankin--Selberg $L$-values in Hida families]{Anticyclotomic main conjecture and the non-triviality of Rankin--Selberg $L$-values in Hida families}

\today
\author{Chan-Ho Kim}
\address{Center for Mathematical Challenges, Korea Institute for Advanced Study, 85 Hoegiro, Dongdaemun-gu, Seoul 02455, Republic of Korea}
\email{chanho.math@gmail.com}

\author{Matteo Longo}
\address{Dipartimento di Matematica, Universit\`a di Padova, Via Trieste 63, 35121 Padova, Italy}
\email{mlongo@math.unipd.it}

\thanks{}

\begin{abstract}
We prove the two-variable anticyclotomic Iwasawa main conjecture for Hida families and discuss its arithmetic application to a definite version of the horizontal non-vanishing conjecture, which is formulated in \cite{LV-MM}.
Our approach is based on the two-variable anticyclotomic control theorem for Selmer groups and
 the relation between the two-variable anticyclotomic $L$-function for Hida families built out of $p$-adic families of Gross points on definite Shimura curves studied in \cite{CL} and \cite{CKL} and the self-dual twist of the specialisation to the anticyclotomic line of the three-variable $p$-adic $L$-function of Skinner--Urban \cite{SU}. 
\end{abstract}



\maketitle

\tableofcontents

\section{Introduction}

To state our main results, fix a prime $p \geq 5$ and an integer $N$ with $p\nmid N$, and let $\mathcal{R}$ be a primitive branch of a Hida family of $p$-adic modular forms of tame conductor $N$; more precisely, $\mathcal{R}$ is a noetherian domain, finite and flat over the Iwasawa algebra $\Lambda=\mathcal{O}\pwseries{1+p\Z_p}$, where $\mathcal{O}$ is the valuation ring of a fixed finite extension of $\Q_p$.
We assume throughout that $\mathcal{R}$ is a regular ring. Let 
$\mathbf{f}=\sum_{n\geq 1}\mathbf{a}_nq^n\in \mathcal{R}\pwseries{q}$ denote the Hida family of $p$-adic modular forms associated with $\mathcal{R}$. 
For each arithmetic prime $\kappa:\mathcal{R}\rightarrow F_\kappa \subseteq \bar\Q_p$, where $F_\kappa$ is a finite extension of $\mathbb{Q}_p$, the specialisation 
$f_\kappa=\sum_{n\geq 1}\kappa(\mathbf{a}_n)q^n \in F_\kappa\pwseries{q}$  of $\mathbf{f}$ at $\kappa$ is a $p$-ordinary cuspform of level $\Gamma_1(Np^{s_\kappa})$, weight $k_\kappa$ and character $\psi_\kappa$, for an integer $s_\kappa\geq 1$ and an integer $k_\kappa\geq 2$; see \S\ref{Hida} for a more accurate exposition. 

Let $\mathbf{T}^\dagger$ denote the self-dual twist of Hida's big Galois representation attached to $\mathcal{R}$; therefore, $\mathbf{T}^\dagger$ is a free $\mathcal{R}$-module of rank $2$, 
and for each arithmetic prime $\kappa:\mathcal{R}\rightarrow F_\kappa$, the specialisation $V_\kappa^\dagger=\mathbf{T}^\dagger\otimes_{\mathcal{R},\kappa}F_\kappa$ at $\kappa$ is isomorphic to the base change to $F_\kappa$ of the self-dual twist of Deligne's Galois representation attached to the modular form $f_\kappa$. Let $\mathbf{A}^\dagger=\mathbf{T}^\dagger\otimes_{\mathcal{R}}\mathcal{R}^\vee$. Here, $(-)^\vee$ means the Pontryagin dual.

Let $K$ be a quadratic imaginary field of discriminant prime to $Np$, and write the factorisation $N=N^+N^-$ where a prime divisor $\ell$ of $N$ divides $N^+$ if and only if it is split in $K$. Throughout the paper, we place ourselves in the \emph{definite setting}; more precisely, we assume that 
\begin{itemize}
\item $N^-$ is a square-free product of an odd number of distinct primes; we also assume that $p$ is split in $K$. 
\end{itemize} 
If $f_\kappa$ has trivial character, the order of vanishing at $s=k/2$ of the $L$-series $L(f_\kappa/K,s)$ of $f_\kappa$ over $K$ is even by the assumption above on $N^-$, and therefore, in light of an analogue in this setting of Greenberg's conjecture, one expects that these $L$-values do not generically vanish when the order of vanishing of the $L$-series $L(f_\kappa,s)$ of $f_\kappa$ over $\Q$ is also even. 
As a consequence of the Tamagawa number conjecture of Bloch--Kato, one 
expects that the Bloch--Kato Selmer groups of $V_\kappa^\dagger$ over $K$ is generically trivial in the same setting, 
so one expects that Nekov\'a\v{r}'s extended Selmer group of $\mathbf{A}^\dagger$ over $K$ is a cotorsion $\mathcal{R}$-module.  when the order of vanishing of the $L$-series $L(f_\kappa,s)$ of $f_\kappa$ over $\Q$ is even. 
See \cite[Conjecture 9.5]{LV-MM} for a more detailed discussion of this heuristic.

Denote by $K_\infty$ the anticyclotomic $\Z_p$-extension of $K$, and denote $\Gamma_\infty=\Gal(K_\infty/K)\simeq\Z_p$. 
In \cite{LV-MM} a big theta element $\Theta_\infty(\mathbf{f})$ 
in $\mathcal{R}\pwseries{\Gamma_\infty}$ 
is constructed by means of compatible families of Gross points on towers of Shimura curves associated with the definite quaternion algebra $B$ ramified at all primes dividing $N^-$ and Eichler orders of increasing $p$-power level; we also define the two-variable anticyclotomic $p$-adic $L$-function 
\[L_p(\mathbf{f}/K)=\Theta_\infty(\mathbf{f})\cdot\Theta_\infty(\mathbf{f})^*,\] where $x\mapsto x^*$ is the involution of $\mathcal{R}\pwseries{\Gamma_\infty}$ defined by $\gamma\mapsto\gamma^{-1}$ on group-like elements. The construction $\Theta_\infty(\mathbf{f})$, $L_p(\mathbf{f}/K)$, and the notion of compatibility of Gross points on towers of Shimura curves, are reviewed in \S\ref{p-adic-L-functions}.

The element $L_p(\mathbf{f}/K)$ is the analogue of the $p$-adic $L$-function $L_p(E/K)$ introduced by Bertolini--Darmon \cite{BD-heegner}, 
\cite{BDmumford-tate} and \cite{BD-IMC} using Gross points on definite Shimura curves to study the arithmetic of elliptic curves over $K$ in a similar definite setting. 
In particular, if the conductor $N$ of $E$ admits the same factorisation $N=N^+N^-$ as above and $p$ is a prime of good ordinary reduction for $E$ which splits in $K$, then, under mild technical hypotheses, it is known that 
the $p$-adic $L$-function $L_p(E/K)$ is a non-trivial element of 
$\Z_p\pwseries{\Gamma_\infty}$, 
the Pontryagin dual of the Selmer group of the $p$-power torsions of $E$ over $K_\infty$ is a torsion $\Z_p \pwseries{\Gamma_\infty}$-module and its characteristic ideal of equal to the ideal generated by the $p$-adic $L$-function 
(see \cite[Theorem 1]{BD-IMC} and \cite[Theorem 3.37]{SU}). 
Therefore, it is natural to expect a similar main conjecture holds for families.
 In other words, if $\mathcal{R}\pwseries{\Gamma_\infty}$ is the Iwasawa algebra of $\Gamma_\infty$ with coefficients in $\mathcal{R}$, then one expects that 
the two-variable $p$-adic $L$-function $L_p(\mathbf{f}/K)$ is a non-zero element of $\mathcal{R}\pwseries{\Gamma_\infty}$, the Selmer group of $\mathbf{A}^\dagger$ over $K_\infty$ is a cotorsion $\mathcal{R}\pwseries{\Gamma_\infty}$-module, and its characteristic ideal is equal to the ideal generated by the $p$-adic $L$-function. See \cite[\S9.3]{LV-MM} for a more detailed discussion of this topic. The proof of these assertions is one of the main result of this paper, from which we also deduce some results on the Selmer group 
of $\mathbf{T}^\dagger$ over $K$. 

Let $\tilde{H}^1_f(K_\infty,\mathbf{A}^\dagger)$ be Nekov\'{a}\v{r}'s extended Selmer group of $\mathbf{A}^\dagger$ over $K_\infty$ and  
$\tilde{H}^1_f(K_\infty,\mathbf{A}^\dagger)^\vee$ its Pontryagin dual, which are discrete and compact $\mathcal{R}\pwseries{\Gamma_\infty}$-modules, respectively.
Before stating our main results, we fix our assumptions.
We suppose that there exists an arithmetic prime $\kappa_0$ such that $f_0=f_{\kappa_0}=\sum_{n\geq 1}a_nq^n \in S_{k_0}(\Gamma_0(Np))$ is an ordinary $p$-stabilised newform of weight $k_0\geq 2$
with $k_0 \equiv 2\pmod{p-1}$ and trivial nebentypus character.
We denote $\bar\rho_{f_0}$ the residual representation attached to $f_0$.

Our first main result is the two-variable anticyclotomic Iwasawa main conjecture for Hida families, which proves \cite[Conjecture 9.12]{LV-MM}. 
\begin{theorem} \label{thm:main-theorem-1}
We assume the following statements.
\begin{itemize}
\item 
$N^-$ is a square-free product of an odd number of distinct primes. 
\item The residual representation $\bar{\rho}_{f_0}$ is absolutely irreducible, $p$-distinguished, and ramified at all primes $\ell\mid N^-$.
\item $p$ is a \emph{non-anomalous prime} for $\bar{\rho}_{f_0}$ when $k=2$, i.e. $a_p(f_0) \not\equiv \pm1$ modulo the maximal ideal of $\mathcal{O}$. 
\item $p$ is split in $K$. 
\end{itemize}
Then
$\tilde{H}^1_f(K_\infty,\mathbf{A}^\dagger)^\vee$ is a cotorsion $\mathcal{R}\pwseries{\Gamma_\infty}$-module, and its characteristic ideal is equal to 
the ideal generated by the two-variable $p$-adic $L$-function $L_p(\mathbf{f}/K)$. 
\end{theorem}

We also deduce a result on the arithmetic of $\mathbf{f}$ over $K$, which is an 
definite analogue of the horizontal non-vanishing conjecture \cite[Conjecture 9.5]{LV-MM}.
 Define 
 $\mathcal{J}_0=\chi_\mathrm{triv}\left(\Theta_{\infty}(\mathbf{f})\right)$, 
 where $\chi_\mathrm{triv}$ is the trivial character of $\Gamma_\infty$ 
 and let $\tilde H^1_f(K,\mathbf{T}^\dagger)$ denote Nekov\'a\v{r}'s extended Selmer group of $\mathbf{T}^\dagger$ over $K_\infty$.
 \begin{theorem}
 Under the same assumptions in Theorem $\ref{thm:main-theorem-1}$, if 
$\tilde H^1_f(K,\mathbf{T}^\dagger)$ is a torsion $\mathcal{R}$-module, then $\mathcal{J}_0\neq0$.  
\end{theorem} 

The proofs of these results are the combination of the following ingredients.  
\begin{itemize}
\item A control theorem for Selmer groups of Hida's big Galois representations over the antiyclotomic $\Z_p$-extension, similar to analogous results for the cyclotomic $\Z_p$-extension by Ochiai \cite{Ochiai3}, which we prove in \S \ref{CT} of this paper;
\item The results from \cite{CL}, \cite{CKL} and \cite{KL} proving a close relation between $L_p(\mathbf{f}/K)$ and the self-dual twist of the specialisation to the anticyclotomic line of the three-variable $p$-adic $L$-functions of Skinner--Urban \cite{SU}; 
\item The three-variable Iwasawa main conjecture proved by Skinner--Urban \cite{SU}. 
\end{itemize}
As hinted at in the lines above, the proof of the three-variable main conjecture in \cite{SU} has a prominent role in our argument; however, the careful comparison of the two setting is required, for which we use the results from \cite{CL}, \cite{CKL} and \cite{KL}. 

\subsubsection*{Acknowledgements} We thank Francesc Castella and Stefano Vigni for useful discussions. 
Kim was partially supported 
by a KIAS Individual Grant (SP054103) via the Center for Mathematical Challenges at Korea Institute for Advanced Study and
by the National Research Foundation of Korea (NRF) grant funded by the Korea government (MSIT) (No. 2018R1C1B6007009). 
Longo was partially supported by  PRIN 2017 \emph{Geometric, algebraic and analytic methods in arithmetic} and INDAM GNSAGA.

\section{Selmer groups over ordinary deformation rings and their control theorem} \label{CT}
In this section, we first review Iwasawa algebras over complete noetherian regular local rings of Krull dimension $\geq 1$ and Selmer groups of ordinary Galois representations over such rings. 
Then we prove a general control theorem for these Selmer groups and relate them with classical Selmer groups via Shapiro's lemma.
This generality certainly includes the case of Hida deformations.  
The notation of this section is independent of the notation of the other sections of the paper. Some of the arguments are similar to those in \cite{Ochiai1}, \cite{Ochiai}, and \cite{Ochiai3}. 

We first set some general conventions. Let $R$ be a complete noetherian regular local ring with maximal ideal $\mathfrak{m}_R$, of Krull dimension $d\geq 1$, with finite residue field $k=R/\mathfrak{m}_R R$ of characteristic $p$, a prime number. 
For any ideal $I\subseteq R$, and any $R$-module $M$, denote $M[I]$ the $I$-torsion $R$-submodule of $M$ and $M_I$ the localization of $M$ at $I$. Denote 
\[M^*=\Hom_R(M,R)\] the $R$-linear dual of $M$ (where $\Hom_R$ 
denotes $R$-linear homomorphisms) and  
\[M^\vee=\Hom_\mathrm{cont}(M,\Q_p/\Z_p)\] 
the Pontryagin dual of $M$ (where $\Hom_\cont$ denotes 
continuous group homomorphisms). 
By \cite[\S 2.9.1, \S 2.9.2]{Nekovar}, 
\[M^\vee=D(M)=\Hom_R(M,R^\vee)\] under our assumptions for any 
$R$-module $M$ of finite type, hence compact, or any $R$-module $M$ of cofinite type equipped with 
the discrete topology. 
Following \cite[\S 0.4]{Nekovar}, define
\[\Phi(M)=M\otimes_R R^\vee.\]
In particular, $(M^*)^\vee\simeq\Phi(M)$ and $(\Phi(M))^\vee\simeq M^*$  for any $R$-module $M$ of finite type 
(\cite[(0.4.4)]{Nekovar}). Further, by basic properties of Pontryagin duality, 
$(M[\p])^\vee\simeq M^\vee/\p M^\vee$ and, if $M$ is a $G$-module 
for some profinite group $G$, we have $(M^G)^\vee\simeq (M^\vee)_G$. 

\subsection{Iwasawa algebras over regular local rings} Fix a complete noetherian regular local ring $R$, with maximal ideal $\mathfrak{m}_R$, of Krull dimension $d\geq 1$, and finite residue field $k=R/\mathfrak{m}_R$ of characteristic $p$, a prime number. 
Let $F_\infty/F$ be a $\Z_p$-extension of $F$, unramified outside $p$ and totally ramified at $p$, and define $G_\infty=\Gal(F_\infty/F)\simeq\Z_p$. 
Let $F_n$ be the subfield of $F_\infty$ such that 
$G_n=\Gal(F_n/F)\simeq\Z/p^n\Z$ and define 
\[\Lambda_R=R[\![G_\infty]\!]=\invlim_nR[G_n].\] 

We recall briefly some properties of $\Lambda_R$ and finitely generated $\Lambda_R$-modules. We begin with the following standard fact.  

\begin{lemma} \label{lemma3.1} The ring 
$\Lambda_R$ is isomorphic to the power series ring 
$R\pwseries{X}$ via the map which sends a topological generator $\gamma$ of $G_\infty$ to $X-1$. 
\end{lemma} 

Since $R$ is a complete noetherian regular local ring, thanks to Lemma \ref{lemma3.1} we see that $\Lambda_R$ is also a complete noetherian regular local ring with maximal ideal  
$\mathfrak{m}_{\Lambda_R}=(\mathfrak{m}_R,\gamma-1)$ of $\Lambda_R$
(\cite[Theorem 3.3, Exercise 8.6, Theorem 19.5]{Mat}).
In particular, since $R$ and $\Lambda_R$ are regular local ring, they are also UFD by 
Auslander--Buchsbaum Theorem (\cite[Theorems 20.3 and 20.8]{Mat}), and therefore 
every prime ideal of height $1$ of $R$ and $\Lambda_R$ is principal (\cite[Theorem 20.1]{Mat}),  
and $R$ and $\Lambda_R$ are integrally closed (\cite[\S 9, Example 1]{Mat}). 

Recall that a $\Lambda_R$-module $X$ is said to be \emph{pseudo-null} if 
its support $\mathrm{Supp}_{\Lambda_R}(X)$ contains only prime ideals of height at least $2$, and that two $\Lambda_R$-modules $X$ and $Y$ are said to be \emph{pseudo-isomorphic} if there exists an exact sequence 
\[0\longrightarrow A\longrightarrow X\longrightarrow Y\longrightarrow B\longrightarrow 0\]
where $A$ and $B$ are pseudo-null $\Lambda_R$-modules (\cite[Chapter VII, \S 4, no.4, Definitions 2 and 3]{Bou}).
Since $\Lambda_R$ is noetherian and integrally closed, we see from \cite[Chapter VII, \S 4, no.4, Theorem 4]{Bou} that every finitely generated $\Lambda_R$-module $M$ is pseudo-isomorphic to a $\Lambda_R$-module of the form $T\times Q$, where $T$ is the maximal torsion $\Lambda_R$-submodule of $M$ and $Q$ is a free $\Lambda_R$-module. 
By  \cite[Chapter VII, \S 4, no.4, Theorem 5]{Bou}, we know that $T$ is isomorphic to $\oplus_{i=1}^t\Lambda_R/\p_i^{n_i}$ 
for suitable height 1 prime ideals $\p_i$ of $\Lambda_R$ and integers $n_i\geq 1$;  
moreover, since every prime ideal of $\Lambda_R$ is principal, there are prime (hence irreducible) elements $g_i\in\Lambda_R$ such that  
$T\simeq\oplus_{i=1}^s \Lambda_R/g_i^{n_i}\Lambda_R$. Define the \emph{characteristic ideal} $\Char_{\Lambda_R}(M)$ 
of $M$ to be $0$ if $Q\neq 0$ and 
\[\Char_{\Lambda_R}(M)=\left(\prod_{i=1}^sg_i^{n_i}\right)\]
otherwise.

\begin{lemma}\label{lemma3.3}
 Let $Q$ be a finitely generated $\Lambda_R$-module and $\p=(g )$ a principal prime ideal of $\Lambda_R$. Assume that $Q/\p Q$ is pseudo-null. Then the $\p$-torsion $\Lambda_R$-submodule 
$Q[\p]$ of $Q$ is isomorphic to $S[\p]$, where $S$ is the maximal pseudo-null $\Lambda_R$-submodule of $Q$. 
\end{lemma}

\begin{proof}
The theory of $\Lambda_R$-modules recalled above shows the existence of an exact sequence 
\begin{equation}\label{pseudo}
0\longrightarrow S\longrightarrow Q\longrightarrow M=U\oplus\left(\bigoplus_{i=1}^s \Lambda_R/g_i^{n_i}\Lambda_R\right)\longrightarrow B\longrightarrow0\end{equation}
where $g_i\in\Lambda_R$ are irreducible elements, $n_i\geq 1$ are integers, $U$ is free over $\Lambda_R$, and $S$ and $B$ are pseudo-null. 
It suffices to show that the multiplication by g map is injective on $M$.

Since $U$ is torsion-free, the multiplication by $g$ map is injective on $U$. 

We now make the following observation. Suppose that $g_i\mid g$ for some $i$. 
Then $\p=( g )=(g_i)$ since $g$ is irreducible.
This,  the quotient ring $\Lambda_R/( g ,g_i^{n_i})\Lambda_R$ is isomorphic to $\Lambda_R/\p\Lambda_R$, which is not pseudo-null, and therefore $M/\p M$ is not pseudo-null. 
Since subquotients of pseudo-null $\Lambda_R$-modules are again pseudo-null, from \eqref{pseudo} we have a pseudo-isomorphism between the pseudo-null $\Lambda_R$-module $Q/\p Q$ 
and the $\Lambda_R$-module $M/\p M$ which is not pseudo-null. 
Hence, $g_i\nmid g$ for every $i$ under our assumption. 

We now study the multiplication by $g$ map on the torsion $\Lambda_R$-submodule of $M$.
Suppose that $g \cdot [m]=0$ for some class $[m]\in\Lambda_R/g_i^{n_i}\Lambda_R$, where $m\in\Lambda_R$. Then $g \cdot m$ belongs to $g_i^{n_i}$. Since $g_i^{n_i}\mid g \cdot  m$ and $g_i\nmid g$, we conclude that 
$g_i^{n_i}\mid m$, so $[m]=0$. Thus the multiplication by $g$ map is injective on $\Lambda_R/g_i^{n_i}\Lambda_R$. We conclude that the multiplication by $g$ map $M\overset {\times g} \rightarrow M$ is injective, and therefore the $\p$-torsion $\Lambda_R$-submodule $Q[\p]$ of $Q$ is isomorphic to the $\p$-torsion $\Lambda_R$-submodule of $S$, as was to be shown. 
\end{proof}

\subsection{Selmer groups over Iwasawa algebras} \label{sec 2.2} 
Let 
$F$ be an algebraic number field. For each place $v$ of $F$, 
denote $F_v$ the completion of $F$ at $v$ and $\mathcal{O}_{F_v}$ 
the valuation ring of $F_v$. Define $G_F=\Gal(\bar{F}/F)$ and 
$G_{F_v}=\Gal(\bar{F}_v/F_v)$. Let $I_{F_v}$ the inertia subgroup 
of $G_{F_v}$. We will also write $\mathcal{O}_v=\mathcal{O}_{F_v}$, $I_v=I_{F_v}$ and $G_v=G_{F_v}$ 
when the fields involved are clear from the context. Recall that $F_\infty/F$ is a fixed $\Z_p$-extension of $F$, unramified outside $p$ and totally ramified at $p$, $G_\infty=\Gal(F_\infty/F)$ and
$F_n$ is the subfield of $F_\infty$ such that 
$G_n=\Gal(F_n/F)\simeq\Z/p^n\Z$; finally, recall that 
$\Lambda_R=R[\![G_\infty]\!]$. 

Let $\T$ be finite free $\Lambda_R$-module equipped with a continuous 
action of $G_F$, and fix a prime number $p$ a prime number. Let $\Sigma_p$ denote the set of places of $F$ dividing $p$, and let $\Sigma$ be a finite set of places of $F$ 
containing $\Sigma_p$. We assume that $\T$ is unramified outside $\Sigma$.  
Moreover, for each $v\mid p$ a prime of the 
ring of integers $\mathcal{O}$ of $F$, we suppose given a filtration 
\begin{equation}\label{filtration}
0\longrightarrow F_v^+(\T)\longrightarrow \T\longrightarrow F^-_v(\T)\longrightarrow0\end{equation}
of $G_v=\Gal(\bar{F}_v/F_v)$-modules.

\begin{remark}
For the moment, we do not impose any condition to the filtration \eqref{filtration}, but of course the structure of the Selmer group defined below depends on this choice. In the applications, the filtration \eqref{filtration} is made of $\Lambda_R$-modules $F_v^+(\T)$ and $ F^-_v(\T)$ which are both free of rank $1$, 
and the Galois action on each of them is characterised by a pair of characters, one unramified and the other factorising through the cyclotomic $\Z_p$-extension of $F$. See \S\ref{section shapiro} for details. 
\end{remark}  

Taking $\Phi$ (\emph{i.e.} tensoring over $\Lambda_R$ with $\Lambda_R^\vee$) we also get a filtration 
\[0\longrightarrow F_v^+(\A)\longrightarrow \A\longrightarrow F^-_v(\A)\longrightarrow0.\]
Define the \emph{Greenberg Selmer group of $\A$} (relative to the chosen filtrations \eqref{filtration}) by
\[\Sel(F,\A)=\ker\left(H^1(F,\A)\longrightarrow \prod_{v\not\in\Sigma_p}H^1(I_v,\A)\times\prod_{v\in\Sigma_p} H^1(I_v,\A/F^+_v(\A))\right)\]
and the \emph{strict Greenberg Selmer group of $\A$} (relative to the chosen filtrations \eqref{filtration}) by
\[\Sel_\mathrm{str}(F,\A)=\ker\left(H^1(F,\A)\longrightarrow \prod_{v\not\in\Sigma_p}H^1(I_v,\A)\times \prod_{v\in\Sigma_p}H^1(F_v,\A/F^+_v(\A))\right)\]
where $I_v$ is the inertia subgroup of $G_v$. 

Let $\q=(g)\subseteq \Lambda_R$ be a principal ideal and assume that $\Lambda_R/\q\Lambda_R$ is finite 
and flat over $R$. Since $\Lambda_R/\q\Lambda_R$ is flat over $R$, tensoring over $R$ with $\Lambda_R/\q\Lambda_R$ we also have a filtration 
\[0\longrightarrow F_v^+(\T/\q \T)\longrightarrow \T/\q \T\longrightarrow F^-_v(\T/\q \T)\longrightarrow0\]
where $\T/\q \T=\T\otimes_{\Lambda_R}\Lambda_R/\q\Lambda_R=T\otimes_R\Lambda_R/\q\Lambda_R$. 
We also have a filtration 
\[0\longrightarrow F_v^+(\A[\q])\longrightarrow \A[\q]\longrightarrow F^-_v(\A[\q])\longrightarrow0\]
where $F_v^+(\A[\q])=A[\q]\cap F^+_v(\A)$. 
Define the \emph{Greenberg Selmer group of $\A[\q]$} (relative to the chosen filtrations \eqref{filtration}) by
\[\Sel(F,\A[\q])=\ker\left(H^1(F,\A[\q])\longrightarrow \prod_{v\not\in\Sigma_p}H^1(I_v,\A[\q])\times \prod_{v\in\Sigma_p}H^1(I_v,\A[\q]/F^+_v(\A[\q]))\right)\]
and the \emph{strict Greenberg Selmer group of $\A$}  (relative to the chosen filtrations \eqref{filtration}) by
\[\Sel_\mathrm{str}(F,\A[\q])=\ker\left(H^1(F,\A[\q])\longrightarrow \prod_{v\not\in\Sigma_p}H^1(I_v,\A[\q])\times \prod_{v\in\Sigma_p}H^1(F_v,\A/F^+_v(\A[\q]))\right).\]


\subsection{The control theorem} Let the notation be as in 
\S\ref{sec 2.2}. Let $\q=(g)\subseteq \Lambda_R$ be a principal ideal and assume that $\Lambda_R/\q\Lambda_R$ is finite 
and flat over $R$.
Then we have canonical maps 
\[r_\q:\Sel(F,\A[\q])\longrightarrow \Sel(F,\A)[\q],\]
\[r_\q^\mathrm{str}:\Sel_\mathrm{str}(F,\A[\q])\longrightarrow \Sel_\mathrm{str}(F,\A)[\q] . \]

\begin{proposition}\label{prop3.4}
Assume that $H^0(F,\A[\q])^\vee$ is a pseudo-null $\Lambda_R$-module. 
Then $\mathrm{ker} ( r_\q )^\vee$ and $ \mathrm{ker} ( r_\q^\mathrm{str} )^\vee$ are also pseudo-null $\Lambda_R$-modules, and are contained in the $\q$-torsion subgroup of the maximal pseudo-null $\Lambda_R$-submodule of $(\T^*)_{G_F}$. 
\end{proposition}
\begin{proof} We explain the argument only for $r_\q$; the case of $r_\q^\mathrm{str}$ is verbatim. 

We have the following commutative diagram:
\[\xymatrix{
& \Sel(F,\A[\q])\ar[r]^-{r_\q}\ar@{^(->}[d]& \Sel(F,\A)[\q] \ar@{^(->}[d]\\
H^0(F,\A)/\q H^0(F,\A)\ar[r] & H^1(F,\A[\q])\ar[r] & H^1(F,\A)[\q] 
}
\]
therefore it is enough to show that $H^0(F,\A)/\q H^0(F,\A)$ is a pseudo-null $\Lambda_R$-module, and that it is contained in the $\q$-torsion subgroup of the maximal pseudo-null $\Lambda_R$-submodule of $(\T^*)_{G_F}$.

Note that $H^0(F,\A[\q])=H^0(F,\A)[\q]$ is the Pontryagin dual of $(\T^*)_{G_F}/\q(\T^*)_{G_F}$, and that 
$H^0(F,\A)/\q H^0(F,\A)$ is the Pontryagin dual of $(\T^*)_{G_F}[\q]$. 
Since $H^0(F,\A[\q])^\vee$ is a pseudo-null $\Lambda_R$-module by assumption, 
applying Lemma \ref{lemma3.3} to the $\Lambda_R$-module 
$(\T^*)_{G_F}$ we see that $H^0(F,\A)/\q H^0(F,\A)$ has also pseudo-null Pontryagin dual, 
contained in the $\q$-torsion subgroup of the maximal pseudo-null $\Lambda_R$-submodule of $(\T^*)_{G_F}$.
\end{proof}

For $v\in\Sigma$, define 
\[C_v=\begin{cases}
\text{$\Lambda_R$-torsion submodule of the module $((\T/F_v^+(\T))^*)_{I_v}$ if $v\in\Sigma_p$},\\
\text{$\Lambda_R$-torsion submodule of the module $(\T^*)_{I_v}$ if $v\in\Sigma-\Sigma_p$.} 
\end{cases}\]
\[C_v^\mathrm{str}=\begin{cases}
\text{$\Lambda_R$-torsion submodule of the module $((\T/F_v^+(\T))^*)_{G_v}$ if $v\in\Sigma_p$},\\
\text{$\Lambda_R$-torsion submodule of the module $(\T^*)_{I_v}$ if $v\in\Sigma-\Sigma_p$.} 
\end{cases}\]

Denote $F_\Sigma$ the maximal extension of $F$ which is unramified outside $\Sigma$.
\begin{proposition}\label{prop3.6}
Assume that 
\begin{itemize}
\item  The $\Lambda_R$-module $C_v/\q C_v$ ($C_v^\mathrm{str}/\q C_v^\mathrm{str}$, respectively) is pseudo-null 
for each $v\in\Sigma$; 
\item $H^0(F_\Sigma/F,\A[\q])^\vee$  is pseudo-null.
\end{itemize}
Then 
$\mathrm{coker} ( r_\q )^\vee$ ($\mathrm{coker} ( r_\q^\mathrm{str} )^\vee$, respectively) is a pseudo-null $\Lambda_R$-module. 
\end{proposition}

\begin{proof}We do the proof only for $r_\q$; 
the case of $r_\q^\mathrm{str}$ is verbatim. 

Recall that 
\[H^1(F_\Sigma/F,\A)=\ker\left(H^1(F,\A)\longrightarrow \prod_{v\not\in\Sigma} H^1(I_v,\A)\right)\]
as a submodule of $H^1(F,\A)$. It follows that there exists a commutative diagram: 
{\footnotesize{\[\xymatrix{
0\ar[r] & 
\Sel(F,\A[\q])\ar[r]\ar[r]\ar[d]^-{r_\q} & 
H^1(F_\Sigma/F,\A[\q])\ar[r]^-{\gamma_\q}\ar[d]^-{s_\q}& 
\prod_{p\in\Sigma_p}H^1(I_v,(\A/F^+_v(\A))[\q])
\times\prod_{v\in\Sigma-\Sigma_p}H^1(F_v,\A[\q])\ar[d]^-{t_\q} \\
0\ar[r]& 
\Sel(F,\A)[\q]\ar[r] & 
H^1(F_\Sigma/F,\A)[\q] \ar[r] & 
\prod_{p\in\Sigma_p}H^1(I_v,(\A/F^+_v(\A)))[\q]
\times\prod_{v\in\Sigma-\Sigma_p}H^1(F_v,\A)[\q]
}
\]}} where the vertical arrows are restriction maps. The multiplication by $g$ map induces an exact sequence 
\[0\longrightarrow \mathbf{A}[\mathfrak{q}]\longrightarrow\mathbf{A}\overset{g}\longrightarrow\mathfrak{q}\mathbf{A}
\longrightarrow 0\] 
which shows that map $s_\mathfrak{q}$ is surjective. Therefore by the snake lemma the cokernel of $r_\mathfrak{q}$ 
is a subquotient of the kernel of $t_\mathfrak{q}$. Therefore, it is enough to show that 
the Pontryagin dual $\ker(t_\mathfrak{q})^\vee$ of $\ker(t_\mathfrak{q})$ is pseudo-null. 
%
%
The module $\ker(t_\q)$ is isomorphic to
{\footnotesize{
\begin{align*}
\ker(t_\mathfrak{a}) &\simeq\mathrm{coker}\left(
\prod_{p\in\Sigma_p}H^0(I_v,\A/F^+_v(\A))
\times\prod_{v\in\Sigma-\Sigma_p}H^0(F_v,\A)\overset g\longrightarrow 
\prod_{p\in\Sigma_p}H^0(I_v,\A/F^+_v(\A))
\times\prod_{v\in\Sigma-\Sigma_p}H^0(F_v,\A)
\right) \\
&  \simeq \ker\left(
\prod_{v\in\Sigma-\Sigma_p}(\T^*)_{I_v}\times\prod_{v\in\Sigma_p}((\T/F^+_v(\T))^*)_{G_v}\overset g\longrightarrow 
\prod_{v\in\Sigma_p}((\T/F^+_v(\T))^*)_{G_v}
\times\prod_{v\in\Sigma-\Sigma_p}(\T^*)_{I_v}\right)^\vee.
\end{align*}
} }  
Hence, the module $\ker(t_\q) $ is equal to $ \left(  \oplus_{v\in\Sigma}C_v[\q] \right)^\vee$, by definition.
On the other hand, $\oplus_{v\in\Sigma}C_v/\q C_v$ is pseudo-null by assumption, 
and therefore Lemma \ref{lemma3.3} applied to the module $\oplus_{v\in\Sigma}C_v[\q]$ completes the proof.
\end{proof}

\begin{theorem}\label{CT} Let $\q=(g)$ be a principal ideal of $\Lambda_R$. 
Assume that 
$H^0(F_\Sigma/F,\A[\q])^\vee$ is pseudo-null 
and that the $\Lambda_R$-module $C_v/\q C_v$ ( $C_v^\mathrm{str}/\q C_v^\mathrm{str}$, respectively) is pseudo-null for each $v\in\Sigma$. 
Then  $\mathrm{ker} ( r_\q )^\vee$ and $\mathrm{coker} ( r_\q )^\vee$ ($\mathrm{ker} ( r_\q^\mathrm{str} )^\vee$ and $\mathrm{coker} ( r_\q^\mathrm{str} )^\vee$, respectively) are pseudo-null $\Lambda_R$-modules. 
\end{theorem} 

\begin{proof} Observe that if $H^0(F_\Sigma/F,\A[\q])^\vee$ is pseudo-null 
the same is true for $H^0(F,\A[\q])^\vee$. 
The result then follows combining Proposition \ref{prop3.4} and Proposition \ref{prop3.6}.\end{proof}

\subsection{Shapiro's Lemma} \label{section shapiro}
Let the notation be as in \S\ref{sec 2.2}; thus, $F$ is a number field 
and $F_\infty/F$ is a $\Z_p$-extension, with finite layers $F_n$, 
totally ramified at $p$ and unramified outside $p$.  
Let $T$ be a finite free $R$-module equipped with a continuous 
action of $G_F=\Gal(\bar{F}/F)$ and fix a filtration 
\begin{equation}\label{filtration1}
0\longrightarrow F_v^+(T)\longrightarrow T\longrightarrow F^-_v(T)\longrightarrow0\end{equation}
of $G_v=\Gal(\bar{F}_v/F_v)$-modules, where $F_v$ is the completion of $F$ at $v$. 
Denote $F_v(\mu_{p^\infty})/F_v$ be the cyclotomic extension of $F_v$, 
where $\mu_{p^\infty}$ is the $p$-divisible group of roots of unity in 
$\bar{F}_v$. 
Let $\Sigma_p$ denote the set of places of $F$ dividing $p$, and let $\Sigma$ be a finite set of places of $F$ containing $\Sigma_p$; denote $F_\Sigma/F$ the maximal 
extension of $F$ which is unramified outside $\Sigma$.  
\begin{assumption}\label{ass} We suppose that the following conditions are satisfied.  
\begin{enumerate}
\item $T$ is unramified outside $\Sigma$. 
\item $H^0(F_\Sigma/F_n,A)^\vee$ is pseudo-null. 
\item Both $F^+(T)$ and $F^-(T)$ are free $R$-modules.  
\item For each $v\mid p$, there are characters $\delta_v,\theta_v:G_v\rightarrow R^\times$ such that 
\begin{itemize}
\item $\delta_v$ is unramified and 
takes the Frobenius $\Frob_v$ to $\delta_v(\Frob_v)=u_v$ 
with $u_v\not\equiv1$ modulo the maximal ideal $\mathfrak{m}_R$ of $R$.
\item $\theta_v$ factors through $G_v\rightarrow \Gal(F_v(\mu_{p^\infty})/F_v)$.
\item $G_v$ acts on $F^-_v(T)$ 
via multiplication by the product $\delta_v\cdot\theta_v$.
\end{itemize}
\end{enumerate}
\end{assumption}

Define 
\[A=\Phi(T)=T\otimes_{R}R^\vee.\]  
The filtration $F^+_v(T)\subseteq T$ induces a filtration 
$F^+_v(A)\subseteq A$ of $A$. For each integer $n\geq0$ and any prime ideal $v$ of $F_n$, let $F_{n,v}$ be the completion of $F_n$ at $v$. 
Denote $\Sigma_{n,p}$ the set of places of $F_{n,v}$ above $p$ and 
define  
\[\Sel_\mathrm{str}(F_n,A)=\ker\left(
H^1(F_n,A)\longrightarrow \prod_{v\not\in\Sigma_{n,p}}H^1(I_{n,v},A)\times \prod_{v\in\Sigma_{n,p}}H^1(F_{n,v},A/F^+_v(A))
\right)\] and 
\[\Sel_\mathrm{str}(F_\infty,A)=\dirlim_n\Sel_\mathrm{str}(F_n,A).\]

For any character $\chi:G_\infty\rightarrow B^\times$, where $B$ is a ring, 
and any $B$-module $M$, let $M(\chi)$ denote the $B$-module $M$ 
equipped with $G_\infty$-action given by 
$g\cdot m=\chi(g)m$. 
Let $\kappa:G_\infty\rightarrow\Lambda_R^\times$ be the tautological character. 
Note in particular that $\Lambda_R(\kappa)$ is just $\Lambda_R$ as $\Lambda_R$-module, but we prefer to keep the notation $\Lambda_R(\kappa)$ to stress that we are considering $\Lambda_R$ as a $\Lambda_R$-module and not as a ring.  
Define the $\Lambda_R$-module 
\[\T=T\otimes_{R}\Lambda_R(\kappa^{-1}).\] Since the extension of rings 
$\Lambda_R/R$ is flat (by Lemma \ref{lemma3.1} and \cite[Exercise 7.4]{Mat}) 
then, tensoring \eqref{filtration} over $R$ with $\Lambda_R$ we also have a filtration 
\[0\longrightarrow F_v^+(\T)\longrightarrow \T\longrightarrow F^-_v(\T)\longrightarrow0\]
where $F^\pm_v(\T)=F^\pm_v(T)\otimes_R\Lambda_R(\kappa^{-1})$.
Define 
\[\A=\Phi(\T)=\T\otimes_{\Lambda_R}\Lambda_R^\vee.\] 

We observe that (\emph{cf.} \cite[\S 2.9.1]{Nekovar})
\[\Lambda_R^\vee=\Hom_\cont(\Lambda_R,\Q_p/\Z_p)\simeq
\Hom_R(\Lambda_R, R^\vee ).\]
Moreover, it we stress the structure of $\Lambda_R$-modules, 
we have  
\[\Lambda_R^\vee(\kappa)\simeq 
\Hom_R(\Lambda_R(\kappa^{-1}), R^\vee ),\qquad
\Lambda_R^\vee(\kappa^{-1})\simeq 
\Hom_R(\Lambda_R(\kappa), R^\vee ).\] 
where we use the standard action of $\Lambda_R$ on $\Hom_R(\Lambda_R, R^\vee )$ given by $(\lambda\cdot\varphi)(x)=\varphi(\lambda^{-1}x)$ for $\lambda\in\Lambda_R$ and 
$\varphi\in\Hom(\Lambda_R, R^\vee )$. 

Note that, since $T$ is a free $R$-module, 
we have isomorphisms of $\Lambda_R$-modules: 
\[\begin{split}
\A&=\T\otimes_{\Lambda_R} \Lambda_R^\vee\\
&=(T\otimes_R\Lambda_R(\kappa^{-1}))\otimes_{\Lambda_R}\Hom_\cont(\Lambda_R(\kappa^{-1}),\Q_p/\Z_p)\\
&=(T\otimes_R\Lambda_R(\kappa^{-1}))\otimes_{\Lambda_R}\Hom_R(\Lambda_R(\kappa^{-1}),R^\vee )\\
&=T\otimes_R\Hom_R(\Lambda_R,R^\vee )\\
&=\Hom_R(\Lambda_R,A)
\end{split}\] 

We now concentrate on ideals $\q_n$ generated by elements $\omega_n=\gamma^{p^n}-1$:
\[\q_n=(\omega_n)=(\gamma^{p^n}-1),\] where $\gamma$ is a topological generator of $G_\infty$. 
We have isomorphisms of $\Lambda_R/\q_n\Lambda_R$-modules 
\[\begin{split}
\A[\q_n]&=\Hom_R(\Lambda_R(\kappa),A)[\q_n]\\
&=\Hom_R(\Lambda_R(\kappa)/\q_n\Lambda_R(\kappa),A)\\
&\simeq\Hom_R(R[G_n],A)
\end{split}
\]

\begin{lemma}\label{Shapiro}
For each integer $n\geq0$ we have 
$\Sel_\mathrm{str}(F,\A[\q_n])\simeq \Sel_\mathrm{str}(F_n,A)$. 
Moreover, we have $\Sel_\mathrm{str}(F,\A)\simeq \Sel_\mathrm{str}(F_\infty,A)$. 
\end{lemma}

\begin{proof}
Shapiro's Lemma shows the the first of the following isomorphism 
\[H^1(F_n,A)\simeq H^1(F,\Hom(R[G_n],A))\simeq H^1(F,\A[\q_n]),\]
while the second follows from the previous discussion. 
Taking direct limits over $n$, we also see that  
\[\begin{split}H^1(F_\infty,A)&=\dirlim_n H^1(F_n,A)\\
&\simeq  \dirlim_nH^1(F,\Hom_R(R[G_n],A))\\
&\simeq
H^1(F,\invlim_n\Hom_R(R[G_n],A))\\
&\simeq
H^1(F,\Hom_R(\Lambda_R(\kappa),A))\\
&\simeq H^1(F,\A)\end{split}
\] where the first and the last isomorphism follow from the previous discussion. 
We need to show that, under these isomorphisms, 
$\Sel_\mathrm{str}(F_n,A)$ corresponds to $\Sel_\mathrm{str}(F,\A[\q_n])$ 
and $\Sel_\mathrm{str}(F_\infty,A)$ corresponds to 
$\Sel_\mathrm{str}(F,\A)$. 

Put $C_n(M)=\Hom_R(R[G_n],M)$ for any $R$-module $M$.  
Let $\Sigma_n$ be the set of places of $F_n$ above places in $\Sigma$. 
We have a commutative diagram: 
{\footnotesize{\[\xymatrix{\Sel_\mathrm{str}(F_n,A)\ar[r]\ar[d]^-{r_n} & 
H^1(F_\Sigma,F_n,A) \ar[r] \ar[d]^-{s_n}&
\prod_{v\in\Sigma_n,v\nmid p}H^1(I_w,A)\times\prod_{v\in\Sigma_n,v\mid p}H^1(F_{n,w},A/F^+_w(A))\ar[d]^-{t_n} \\
\Sel_\mathrm{str}(F,C_n(A))\ar[r]&
H^1(F_\Sigma/F,C_n(A))\ar[r] & 
\prod_{v\in\Sigma,v\nmid p}H^1(I_v,C_n(A))\times\prod_{v\in\Sigma,v\mid p}H^1(F_v,C_n(A/F^+_v(A))) 
}
\]}} 
where $\Sel_\mathrm{str}(F,C_n(A))$ is defined by the 
exactness of the lower horizontal arrow. We claim that the vertical arrow $t_n$ is injective. 
To show this, note that the map $t_n$ is the product local maps $t_{n,v}$ for all $v\in \Sigma_n$, 
so we study first these maps $t_{n,v}$. 
If $w\nmid p$, then $I_w=I_v$ because $F_n/F$ is unramified outside $p$;
the map 
$t_{n,v}$ defined by 
\[t_{n,v}:\prod_{w\mid v}H^1(I_w,A)=H^1(I_v,A)^{\sharp{w\mid v}}\longrightarrow 
H^1(I_v,C_n(A))\simeq H^1(I_v,\Hom_R(R,A))^{\sharp\{w\mid v\}}.\]
It follows that $t_{n,v}$ is injective. The map $t_{n,v}$ for $v\mid p$ 
is defined by 
\[t_{n,v}:H^1(F_{n,w},A/F^+_v(A))\longrightarrow 
H^1(F_v,\Hom_R(R[G_n],A/F_v^+(A)))\]
which are all isomorphisms by Shapiro's Lemma because, being 
$p$ totally ramified in the extension $F_n/F$, we have $\Gal(F_n/F)\simeq G_n$. We therefore conclude that 
$t_n$ is injective.  
Since $s_n$ is an isomorphism, the map $r_n$ is an isomorphism too, showing the result. 
\end{proof}

\begin{lemma}\label{lemma support} Let $M$ be an $R$-module equipped with a $G_\infty$-action, 
denote $M_{R\text{-{tors}}}$ the $R$-torsion submodule of $M$ and let 
\[N=M_{R\text{-tors}}\otimes_R\Lambda_R(\kappa).\] 
Then the quotient $N/(\gamma^{p^n}-1)N$ is a pseudo-null $\Lambda_R$-module for each integer $n\geq 1$. 
\end{lemma}

\begin{proof} 
Set $I=(\gamma^{p^n}-1)$ for convenience.
The support of $N/I N$ consists of the prime ideals of $\Lambda_R$ containing $I$. 
Fix such a height one prime ideal $\mathfrak{a}=(a)$.
Then $a$ is an irreducible factor of $\gamma^{p^n}-1$.
Therefore, $\mathfrak{a}\cap R=0$. 
Thus, we have $\mathfrak{m}_R \setminus \lbrace 0 \rbrace \subseteq (\Lambda_R)^\times_\mathfrak{a} $.
It implies that the localization $(M_{R\text{-tors}})_\mathfrak{a}$ of $M_{R\text{-tors}}$ at $\mathfrak{a}$ is trivial. 
It follows that no height one prime ideal of $\Lambda_R$ lies in the support of $N/IN$, and therefore $N/IN$ is pseudo-null over $\Lambda_R$.   
\end{proof}

\begin{proposition} \label{prop 2.10} 
The Pontryagin duals of the kernel and cokernel of the restriction 
\[\res_{F_\infty/F_n}:\Sel_\mathrm{str}(F_n,A)\longrightarrow \Sel_\mathrm{str}(F_\infty,A)^{\Gal(F_\infty/F_n)}\]
are pseudo-null $\Lambda_R$-modules. 
\end{proposition}

\begin{proof}
We only need to check that 
the assumptions in Theorem \ref{CT} are satisfied.
If so,  the result follows 
by taking $\q_n=(\gamma^{p^n}-1)$ in Theorem \ref{CT}, 
and by using Lemma \ref{Shapiro} to identify 
$\Sel_\mathrm{str}(F,\A[\q_n])$ and $\Sel_\mathrm{str}(F,\A)$ with 
$\Sel_\mathrm{str}(F_n,A)$ and $\Sel_\mathrm{str}(F_\infty,A)$, 
respectively. 

By Shapiro's Lemma, we have 
$H^0(F_\Sigma/F,\A[\q_n])\simeq H^0(F_\Sigma/F_n,A)$, 
and therefore the first assumption in Theorem \ref{CT} is equivalent to (2) in Assumption \ref{ass}.  

We first consider 
$C_v^\mathrm{str}/(\gamma^{p^n}-1)C_v^\mathrm{str}$
for $v\nmid p$.
The action of $I_v$ on $\Lambda_R(\kappa^{-1})$ trivial since all the primes outside $p$ are unramified in $F_\infty$. 
Therefore, $(\T^*)_{I_v}=(T^*)_{I_v}\otimes_R\Lambda_R(\kappa)$, 
and \[C_v^\mathrm{str}=((T^*)_{I_v})_{R\text{-tors}}\otimes\Lambda_R(\kappa)\]
where $((T^*)_{I_v})_{R\text{-tors}}$ is the $R$-torsion submodule of $(T^*)_{I_v}$. 
Thus, for $v\nmid p$, the statement in the assumption of Theorem \ref{CT} 
is equivalent to that \[((T^*)_{I_v})_{R\text{-tors}}\otimes\Lambda_R(\kappa)/(\gamma^{p^n}-1)((T^*)_{I_v})_{R\text{-tors}}\otimes\Lambda_R(\kappa)\] is pseudo-null, which follows from Lemma \ref{lemma support} applied to $M=(T^*)_{I_v}$. 


We now consider 
$C_v^\mathrm{str}/(\gamma^{p^n}-1)C_v^\mathrm{str}$ 
for $v\mid p$. Since $\Lambda_R$ is flat over $R$, we have 
\[(\T/F_v^+(\T))^*=(T/F^+_v(T))^*\otimes_R\Lambda_R(\kappa)=
F^-_v(T)^*\otimes_R\Lambda_R(\kappa).\]
We have
\[
\begin{split}(T/F^+_v(T))^*\otimes_R\Lambda_R(\kappa)&\simeq
\left((T/F^+_v(T))^*\otimes_RR(\theta_v)\right)
\otimes_R\left(\Lambda_R(\kappa)\otimes_RR(\theta_v^{-1})\right)\\
&\simeq\left((T/F^+_v(T))\otimes_RR(\theta_v^{-1})\right)^*
\otimes_R\left(\Lambda_R(\kappa\cdot\theta_v^{-1})\right)\\
&\simeq \left(F^-_v(T)\otimes_RR(\theta_v^{-1})\right)^*
\otimes_R\left(\Lambda_R(\kappa\cdot\theta_v^{-1})\right).\end{split}
\]
The action of $I_v$ on
$F^-_v(T)\otimes R(\theta_v^{-1})$ is trivial by (4) in Assumption \ref{ass},  
and therefore the $I_v$-coinvariant of 
$(\T/F_v^+(\T))^*$ is
\[
(F^-_v(T)\otimes_RR(\theta_v^{-1}))^*\otimes_R\left(
\Lambda_R(\kappa\cdot\theta_v^{-1})\right)_{I_v}.
\] 
Since $\delta_v$ is unramified by (4) in Assumption \ref{ass}, the coinvariant of the action of 
$G_v/I_v$ on 
$\left(F^-_v(T)\otimes_RR(\theta_v^{-1})\right)^*$ is given by 
\[\begin{split}
\frac{\left(F^-_v(T)\otimes_RR(\theta_v^{-1})\right)^*}{(\Frob_v-1)
\left(F^-_v(T)\otimes_RR(\theta_v^{-1})\right)^*}&
\simeq
\frac{\left(F^-_v(T)\otimes_RR(\theta_v^{-1})\right)^*}{(u_v-1)
\left(F^-_v(T)\otimes_RR(\theta_v^{-1})\right)^*}\\
&\simeq
\frac{\left(F^-_v(T)\otimes_RR(\theta_v^{-1})\right)^*}{
\left((u_v-1)F^-_v(T)\otimes_RR(\theta_v^{-1})\right)^*}
\end{split}
\]
By (4) in Assumption \ref{ass}, $u_v$ is not congruent to 1 modulo the maximal ideal of $R$, 
so $u_v-1\in R^\times$, and therefore $(u_v-1)F_v^-(T)=0$. 
Moreover, $\delta_v$ acts trivially on
$\left(\Lambda_R(\kappa\cdot\theta_v^{-1})\right)_{I_v}$.  
Therefore, the $G_v$-coinvariant of $(\T/F_v^+(\T))^*$ is trivial, and it follows in particular that the assumption on $C_v^\mathrm{str}/(\gamma^{p^n}-1)C_v^\mathrm{str}$ 
for $v\mid p$ in Theorem \ref{CT} is satisfied. 
\end{proof}

\begin{lemma}\label{lemma 2.9}
Suppose that $M$ is a pseudo-null $\Lambda_R$-module. 
Then for each integer 
$n\geq 0$, $M/(\gamma^{p^n}-1)M$  
is torsion over $R[G_n] \simeq \Lambda_R/(\gamma^{p^n}-1)\Lambda_R$. 
\end{lemma}
\begin{proof}
Suppose $M/(\gamma^{p^n}-1)M$ is not a torsion $R[\Gamma_n]$-module, and take a copy $N$ of $R[\Gamma_n]$ in $M/(\gamma^{p^n}-1)M$. 
Take any height one prime ideal $\mathfrak{a} = (a)$ of $\Lambda_R$ such that $(a, \gamma^{p^n}-1)=1$. 
Then $N_\mathfrak{a}\neq 0$. In particular, $M_\mathfrak{a}/(\gamma^{p^n}-1)M_\mathfrak{a}\neq 0$ so $M_\mathfrak{a}\neq 0$, which contradicts the assumption that $M$ is a pseudo-null $\Lambda_R$-module.
\end{proof}

\begin{corollary}\label{coro 2.10} 
The Pontryagin duals of the kernel and cokernel of the restriction 
\[\res_{F_\infty/F_n}:\Sel_\mathrm{str}(F_n,A)\longrightarrow \Sel_\mathrm{str}(F_\infty,A)^{\Gal(F_\infty/F_n)}\]
are cotorsion $R[G_n]$-modules. 
\end{corollary}

\begin{proof}
It follows from Proposition \ref{prop 2.10} and Lemma \ref{lemma 2.9}. 
\end{proof}

\section{Anticyclotomic Iwasawa theory for Hida families}
\subsection{Ordinary families of modular forms}\label{Hida}
Let $f_0=\sum_{n=1}^\infty a_nq^n\in S_{k_0}(\Gamma_0(Np))$ an ordinary
$p$-stabilized newform (in the sense of \cite[Def.~2.5]{GS}) of weight $k_0\geq 2$ and trivial nebentypus,
defined over a finite extension $L/\Q_p$. 
Let $\mathcal{O}=\mathcal O_L$ be the valuation ring of $L$ and 
$a_p\in\mathcal{O}^\times$, and
$f_0$ is either a newform of level $Np$, or arises from a newform of level $N$.
Denote\[\rho_{f_0}:G_\Q:={\rm Gal}(\overline{\Q}/\Q)\longrightarrow{\rm GL}_2(\mathcal{O})\] 
the Galois representation
associated with $f_0$. 
Since $f_0$ is ordinary at $p$, the restriction of $\rho_{f_0}$ to a decomposition group
$D_p\subset G_\Q$ is upper-triangular. We also denote $k=k_L$ the residue field of $L$ and 
\[\bar\rho_{f_0}:G_\Q\longrightarrow{\rm GL}_2(k)\] 
the residual representation obtained by reduction modulo the maximal ideal $\mathfrak{m}=\mathfrak{m}_L$ of $\mathcal{O}$. 

\begin{assumption}\label{ass1} The representation
$\bar{\rho}_{f_0}$ is absolutely irreducible, and \emph{$p$-distinguished}, i.e.,
writing $\bar{\rho}_{f_0}\vert_{D_p}\sim\left(\begin{smallmatrix}\bar{\varepsilon}&*\\0&\bar{\delta}\end{smallmatrix}\right)$,
we have $\bar{\varepsilon}\neq\bar{\delta}$.
\end{assumption}

Let $\mathfrak h^{\ord}$ be the Hida ordinary Hecke algebra of tame level $\Gamma_0(N)$, and let $\I$ be the branch of $\mathfrak h^\ord$ passing through ${f_0}$. If $\Lambda:=\mathcal{O}\pwseries{\Gamma}$, where $\Gamma=1+p\Z_p$, then 
$\I$ is a finite flat extension of $\Lambda$ (the structure of $\Lambda$-algebra in $\mathfrak h^\ord$ is given by the action of diamond operators in $\Gamma$). 
The eigenform $f_0$ defines an $\mathcal{O}_L$-algebra homomorphism $\lambda_{f_0}:\I\rightarrow\mathcal{O}$, which is called arithmetic. More generally, an \emph{arithmetic point} of $\I$ is a continuous $\mathcal{O}_L$-algebra homomorphism $\I\overset\kappa\rightarrow\overline{\Q}_p$
such that the composition
\[
\Gamma\longrightarrow\Lambda^\times\longrightarrow\I\xrightarrow{\;\k\;}\overline{\Q}_p^\times
\]
is given by $\gamma\mapsto\psi(\gamma)\gamma^{k-2}$,
for some integer $k\geq 2$ and some finite order character $\psi:\Gamma\rightarrow\overline{\Q}_p^\times$.
We then say that $\k$ has \emph{weight} $k$, \emph{character} $\psi$, and \emph{wild level} $p^m$,
where $m>0$ is such that ${\rm ker}(\psi)=1+p^m\Z_p$.
Denote by $\mathcal X(\I)$ the set of continuous $\mathcal{O}$-algebra homomorphisms 
from $\I$ into $\mathcal O$, and by 
$\mathcal{X}_{\rm arith}(\I)$ the subset of $\mathcal X(\I)$ consisting of 
arithmetic primes. 
For each $\k\in\mathcal{X}_{\rm arith}(\I)$, let $F_\k$ be the residue field of ${\rm ker}(\k)\subset\I$, which is a finite extension of $\Q_p$.

For each $n\geq 1$, let $\mathbf{a}_n\in\I$ be the image of $T_n\in\mathfrak{h}^{\rm ord}$ under the natural
projection $\mathfrak{h}^{\rm ord}\rightarrow\I$, and form the $q$-expansion
\[\F=\sum_{n=1}^\infty\mathbf{a}_nq^n\in\I\pwseries{q}.\] By \cite[Thm.~1.2]{hida86b}, if $\k\in\mathcal{X}_{\rm arith}(\I)$
is an arithmetic prime of weight $k\geq 2$, character $\psi$, and wild level $p^m$, then
\[
f_\k:=\sum_{n=1}^\infty\k(\mathbf{a}_n)q^n\in F_\k[[q]]\]
is (the $q$-expansion of) an ordinary $p$-stabilized newform 
in $S_k(\Gamma_0(Np^{m}),\omega^{k_0-k}\psi)$ 
of level $\Gamma_0(Np^n)$, character $\omega^{k_0-k}\psi$ and weight $k$, 
where $\omega:(\Z/p\Z)^\times\rightarrow\Z_p^\times$ is the Teichm\"uller character.

\subsection{Critical characters} Following \cite[Def.~2.1.3]{howard-invmath}, factor the $p$-adic cyclotomic character as
\[
\varepsilon_{\rm cyc}=\varepsilon_{\rm tame}\cdot\varepsilon_{\rm wild}:G_\Q\longrightarrow \Z_p^\times\simeq\boldsymbol{\mu}_{p-1}\times\Gamma,
\]
and define the \emph{critical character} $\Theta:G_\Q\rightarrow\I^\times$ by
\begin{equation}\label{def:crit}
\Theta(\sigma)=\varepsilon_{\rm tame}^{\frac{k_0-2}{2}}(\sigma)\cdot[\varepsilon^{1/2}_{\rm wild}(\sigma)],
\end{equation}
where $\varepsilon_{\rm tame}^{\frac{k_0-2}{2}}:G_\Q\rightarrow\boldsymbol{\mu}_{p-1}$ is any fixed choice
of square-root of $\varepsilon_{\rm tame}^{k_0-2}$ (see \cite[Rem.~2.1.4]{howard-invmath}),
$\varepsilon_{\rm wild}^{1/2}:G_\Q\rightarrow\Gamma$ is
the unique square-root of $\varepsilon_{\rm wild}$ taking values in $\Gamma$,
and $[\cdot]:\Gamma\rightarrow\Lambda^\times\rightarrow\I^\times$ is the map given by
the inclusion as group-like elements.

Define the character $\theta:\Z_p^\times\rightarrow\I^\times$ by the relation
$
\Theta=\theta\circ\varepsilon_{\rm cyc},
$
and for each $\k\in\mathcal{X}_{\rm arith}(\I)$, let $\theta_\k:\Z_p^\times\rightarrow\overline{\Q}_p^\times$
be the composition of $\theta$ with $\k$. If $\k$ has weight $k\geq 2$ and character $\psi$,
then 
\begin{equation}\label{eq:theta}
\theta_\k^2(z)=z^{k-2}\omega^{k_0-k}\psi(z)
\end{equation}
for all $z\in\Z_p^\times$.

\subsection{$p$-adic $L$-functions} \label{p-adic-L-functions}
Let $K/\Q$ be an imaginary quadratic field of discriminant prime to $Np$. 
Write $N=N^+N^-$, where all primes dividing $N^+$ are split in $K$, and 
all primes dividing $N^-$ are inert in $K$. We will work under the following 

\begin{assumption}\label{ass2}\begin{enumerate}
\item 
$N^-$ is a square-free product of an odd number of primes. 
\item The residual representation $\bar{\rho}_{f_0}$ is ramified at all primes $\ell\mid N^-$.
\item 
$a_p\not\equiv \pm1$ modulo the maximal ideal of $\mathcal{O}$
(we say that  $p$ is a \emph{non-anomalous prime}
for $\bar{\rho}_{f_0}$ in this case). 
\item $p$ is split in $K$. 
\end{enumerate}
\end{assumption} 

Let $B$ be the definite quaternion algebra over $\Q$ of discriminant $N^-$. 
For each prime $\ell\nmid N^-$, fix isomorphisms $\iota_\ell:B\otimes_\Q\Q_\ell\simeq\M_2(\Q_\ell)$. 
Let $m\mapsto R_m$, for $m\geq 0$ an integer, be the sequence of Eichler orders of level $N^+p^m$, 
defined by the condition that
$\iota_\ell(R_m\otimes_\Z\Z_\ell)$ consists of the matrices in $\M_2(\Z_\ell)$ which 
are upper triangular modulo $\ell^{\val_\ell(N^+p^m)}$ for all primes $\ell\nmid N^-$
(thus, in particular,  $R_{m+1}\subseteq R_m$ for all integers $m\geq0$). For a ring $A$, 
denote $\hat{A}$ its profinite completion. 
Let $U_m\subset\widehat{R}_m^\times$ be the compact open subgroup
defined by
\[
U_m:=\left\{(x_q)_q\in\widehat{R}_m^\times\;\;\vert\;\;i_p(x_p)\equiv\mat 1*0*\pmod{p^m}\right\}.
\]
Consider the double coset spaces
\[
\widetilde X_m(K)=B^\times\big\backslash\bigl(\Hom_\Q(K,B)\times\widehat{B}^\times\bigr)\big/U_m,
\]
where $b\in B^\times$ act on left on $(\Psi,g)\in\Hom_\Q(K,B)\times\widehat B^\times$ by
$
b\cdot(\Psi,g)=(bgb^{-1},bg),
$
and $U_m$ acts on $\widehat{B}^\times$ by right multiplication and on $\Hom_\Q(K,B)$ trivially. The space $\widetilde{X}_m(K)$ is equipped with  a nontrivial Galois action defined as follows: If $\sigma\in{\rm Gal}(K^{\rm ab}/K)$ and $P\in\widetilde{X}_m(K)$
is the class of a pair $(\Psi,g)$, then
$
P^\sigma:=[(\Psi,g\widehat{\Psi}(a))],
$
where $a\in K^\times\backslash\widehat{K}^\times$ is such that ${\rm rec}_K(a)=\sigma$,
and we extend this to an action of $G_K$ by letting each
$\sigma\in G_K$ act on $\widetilde{X}_m(K)$ as $\sigma\vert_{K^{\rm ab}}$. 
The space $\widetilde{X}_m(K)$ is also equipped with standard action of Hecke operators $T_\ell$ for $\ell\nmid Np$, 
$U_p$ and diamond operators $\langle d \rangle$ for $d\in\Z_p^\times$. 

Let $D_m:=\Div(\tilde X_m)\otimes\mathcal O_L$ be the divisor group of $\tilde X_m$ and 
denote $\alpha_m:D_m\twoheadrightarrow D_{m-1}$ the canonical projection.
Passing to the ordinary part $D_m^\ord$ 
and tensoring with the primitive component $\I$ gives Hecke modules 
$\mathbf D_m$ (for $m\geq 0$) and, twisting the Galois action by $\Theta^{-1}$, Hecke modules
$\mathbf D_m^\dagger$. The analogous Hecke modules obtained from the inverse limits of the divisor group $D_m$ (with respect to the canonical projection maps $\alpha_m$) are the Hecke modules denoted $\mathbf D$ and $\mathbf D^\dagger$
in \cite[\S  6.4]{LV-MM}. Let $e^\ord$ denote the ordinary projector. 
Denote ${\rm Pic}(\widetilde{X}_m)$ 
the Picard group of $\widetilde{X}_m$. Define the Hecke modules 
$
J_m^{\rm ord}:=e^{\rm ord}({\rm Pic}(\widetilde{X}_m)\otimes_{\Z}\mathcal{O}_L)$, 
$\mathbf{J}_m:=J_m^{\rm ord}\otimes_{\mathfrak{h}^{\rm ord}}\I$ 
and $\mathbf{J}_m^\dagger:=\mathbf{J}_m\otimes_{\I}\I^\dagger$. 
Finally define $\mathbf{J}^\dagger:=\varprojlim_m\mathbf{J}_m^\dagger$. 
The projections ${\rm Div}(\widetilde{X}_m)\rightarrow{\rm Pic}(\widetilde{X}_m)$ induce a map
\[\lambda:\mathbf{D}^\dagger\longrightarrow\mathbf{J}^\dagger.\]

Thanks to Assumptions \ref{ass1} and \ref{ass2}, we have 
${\rm dim}_{k_\I}(\mathbf{J}^\dagger/\mathfrak{m}_\I\mathbf{J}^\dagger)=1$ by \cite[Theorem 3.1]{CKL}; 
here, $\mathfrak{m}_\I$ is the maximal ideal of $\I$, and $k_\I:=\I/\mathfrak{m}_\I$ is its residue field.
By \cite[Prop.~9.3]{LV-MM}, we conclude that the module $\mathbf{J}^\dagger$ is free of
rank one over $\I$. Fix an isomorphism \[\eta:\mathbf J^\dagger\simeq \I.\] 

Let $K_\infty$ be the anticyclotomic $\Z_p$-extension of $K$, 
and define $\Gamma_\infty=\Gal(K_\infty/K)\simeq\Z_p$. Denote $K_n$ the subfield 
of $K_\infty$ such that $\Gamma_n=\Gal(K_n/K)\simeq\Z/p^n\Z$. Define 
\[\Lambda_{\mathcal{R}}=\I\pwseries{\Gamma_\infty}=\invlim_n\I[\Gamma_n].\]

The paper \cite{LV-MM} introduces for each integer $n\geq 0$ a sequence $m\mapsto \tilde{P}_{p^n,m}$ of Gross-Heegner points in $\tilde{X}_m(K)$ of conductor $p^{m+n}$; these points satisfy norm-relations and allows to construct 
big theta elements $\Theta_{n}(\mathbf{f})\in \mathbf{D}[\Gamma_n]$ 
by an inverse limit procedure inverting the $U_p$ operator; we will view 
$\Theta_{n}(\mathbf{f})$ as elements in $\I[\Gamma_n]$ by means of the map 
 $\mathbf{D}\overset\lambda\rightarrow\mathbf{J}\overset\eta\simeq\I$. 
The elements $\Theta_n(\mathbf{f})$ are compatible
under the natural maps $\I[\Gamma_m]\rightarrow\I[\Gamma_n]$ for all $m\geq n$, thus defining an element
$
\Theta_{\infty}(\mathbf{f}):=\varprojlim_n\Theta_{n}(\mathbf{f})
$
in the completed group ring $\Lambda_{\mathcal{R}}$.

\begin{definition}\label{p-adic L-function}
The \emph{two-variable $p$-adic $L$-function} attached $\mathbf{f}$ and $K$ is
the element
\[
L_p(\mathbf{f}/K):=\Theta_{\infty}(\mathbf{f})\cdot\Theta_{\infty}(\mathbf{f})^*
\in\Lambda_{\mathcal{R}},
\]
where $x\mapsto x^*$ is the involution on $\I\pwseries{\Gamma_\infty}$ given by
$\gamma\mapsto\gamma^{-1}$ on group-like elements.
\end{definition}

\subsection{Selmer groups of Hida families} 
Let $\mathbf{T}$ be Hida's big Galois representation associated with $\I$. 
Then $\mathbf{T}$ is a free $\I$-module of rank $2$, equipped with a continuous action of $G_\Q=\Gal(\bar\Q/\Q)$ and a filtration of $\mathcal{R}[G_{\Q_p}]$-modules 
\[0\longrightarrow F^+_v(\mathbf{T})\longrightarrow\mathbf{T}\longrightarrow F^-_v(\mathbf{T})\longrightarrow0\] 
where $G_{\Q_p}=\Gal(\bar{\Q}_p/\Q_p)$ is a decomposition group of $G_\Q$ at $p$. 
Both $F^+_v(\mathbf{T})$ and $F^-_v(\mathbf{T})$
are free $\I$-modules of rank $1$; $G_v$ acts on $F^-_v(\mathbf{T})$ 
via the unamified character $\eta_v:G_v/I_v\rightarrow \I^\times$ which 
takes the arithmetic Frobenius to $U_p$, and $G_v$ acts on $F_v^+(\mathbf{T})$ via $\eta_v^{-1}\varepsilon_\mathrm{cyc}[\varepsilon_\mathrm{cyc}]$. 

Denote $\mathbf{T}^\dagger=\mathbb{T}\otimes\Theta^{-1}$ the critical twist of $\mathbf{T}$ corresponding to 
the choice of the critical character $\Theta$ in \eqref{def:crit}.  
For each arithmetic point, define $F_\kappa=\mathcal{R}_\kappa/\ker(\kappa)\mathcal{R}_\kappa$, where $\mathcal{R}_\kappa$ is the localisation of $\mathcal{R}$ at $\kappa$. Then $V_{\kappa}^\dagger=\mathbf{T}^\dagger\otimes_\mathcal{R}F_\kappa$ is isomorphic to the self-dual twist of Deligne representation $V_{f_\kappa}$ attached to the eigenform $f_\kappa$. If $\p=\p_\kappa=\ker(\kappa)$, we also denote 
$\I_\kappa$ by $\I_\p$ and 
$V_\kappa^\dagger$ by $V_\p^\dagger$. 
Moreover, we have a filtration 
 $\mathcal{R}[G_{\Q_p}]$-modules 
\[0\longrightarrow F^+_v(\mathbf{T}^\dagger)\longrightarrow\mathbf{T}^\dagger\longrightarrow F^-_v(\mathbf{T}^\dagger)\longrightarrow0\] 
where $G_v$ acts on $F^-_v(\mathbf{T}^\dagger)$ 
via the character $\eta_v\Theta^{-1}$ and $G_v$ acts on $F_v^+(\mathbf{T}^\dagger)$ via $\eta_v^{-1}\Theta^{-1}\varepsilon_\mathrm{cyc}[\varepsilon_\mathrm{cyc}]$. 
Let \[\mathbf{A}^\dagger=\Phi(\mathbf{T}^\dagger)=\mathbf{T}^\dagger\otimes_\I\I^\vee.\] 

As in \S \ref{sec 2.2} we introduce strict Greenberg Selmer groups 
$\Sel_\mathrm{str}(K_n,\mathbf{A}^\dagger)$ and 
$\Sel_\mathrm{str}(K_\infty,\mathbf{A}^\dagger)$ and 
Selmer groups $\Sel(K_n,\mathbf{A}^\dagger)$ and 
$\Sel(K_\infty,\mathbf{A}^\dagger)$. Under our assumptions, 
by \cite[Theorem 4.1]{CKL}, 
we know that 
$\Sel_\mathrm{str}(K_n,\mathbf{A}^\dagger)\simeq\Sel(K_n,\mathbf{A}^\dagger)$ and $\Sel_\mathrm{str}(K_\infty,\mathbf{A}^\dagger)\simeq\Sel(K_\infty,\mathbf{A}^\dagger)$. 
We may also consider Nekov\'a\v{r}'s extended Selmer groups $\tilde H^1_f(K_n,\A^\dagger)$ 
and $\tilde H^1_f(K_\infty,\A^\dagger)$. By \cite[Lemma 9.6.3]{Nekovar} 
we have an exact sequence 
\[H^0(K_n,\A^\dagger)\longrightarrow \tilde H^1_f(K_n,\A^\dagger)\longrightarrow \Sel_\mathrm{str}(K_n,\A^\dagger)\longrightarrow 0.\]

\begin{lemma} \label{vanishing}
$H^0(K_n,\A^\dagger)=0$.
\end{lemma}

\begin{proof} Let $M=H^0(K_n,\A^\dagger)^\vee$ be the Pontryagin dual 
of $H^0(K_n,\A^\dagger)$. By the topological Nakayama's Lemma, 
it is enough to show that $M/\mathfrak{m}_\I M=0$, where $\mathfrak{m}_\I$ 
is the maximal ideal of $\I$. For this, taking again Pontryagin duals, it is enough to show that 
\[H^0(K_n,\A^\dagger)[\mathfrak{m}_R]=H^0(K_n,\A^\dagger[\mathfrak{m}_R])=0.\] Now the Galois representation $\A^\dagger[\mathfrak{m}_R]$ is isomorphic to $\bar\rho_{f_0}$, which is irreducible by assumption, and it follows from standard arguments (\emph{e.g.} \cite[Lemmas 3.9, 3.10]{LV-TAMS}) that the $K_n$-invariants of $\A^\dagger[\mathfrak{m}_\I]$ are trivial. 
\end{proof}

It follows from Lemma \ref{vanishing} that 
$\tilde H^1_f(K_n,\A^\dagger)\simeq\Sel_\mathrm{str}(K_n,\A^\dagger)$. 
Thus, summing up, we 
have 
\begin{equation}\label{selmer comparison}
\Sel_\mathrm{str}(K_n,\A^\dagger)\simeq\Sel(K_n,\A^\dagger)\simeq
\tilde H^1_f(K_n,\A^\dagger)\end{equation}
and, taking direct limits with respect to the canonical restriction maps,  
\begin{equation}\label{selmer comparison infty}
\Sel_\mathrm{str}(K_\infty,\A^\dagger)\simeq\Sel(K_\infty,\A^\dagger)\simeq
\tilde H^1_{f,\mathrm{Iw}}(K_\infty,\A^\dagger)=\dirlim_n\tilde H^1_f(K_n,\A^\dagger).\end{equation}

\subsection{Control theorems for Hida representations}
Let $I_n$ be the kernel of the map $\Lambda\rightarrow\mathcal{O}$ 
which takes the topological generator $\gamma$ of $\Gamma_\infty$ 
to $\gamma^{p^n}-1$. For an integer $n\geq 0$, define 
\[\Delta_n=\Gal(K_\infty/K_n).\] In particular, we have 
$\Gamma_\infty/\Delta_n\simeq\Gamma_n$.  

\begin{theorem}\label{control theorem}
The kernel and cokernel of the map 
\[\tilde H^1_{f}(K_n,\mathbf{A}^\dagger)
\longrightarrow \tilde H^1_{f,\mathrm{Iw}}(K_\infty,\mathbf{A}^\dagger)^{\Delta_n}\]
are cotorsion $\Lambda_{\mathcal{R}}/I_n\Lambda_{\mathcal{R}}\simeq\I[\Gamma_n]$-modules.\end{theorem}

\begin{proof}
This follows from Corollary \ref{coro 2.10} and \eqref{selmer comparison}, \eqref{selmer comparison infty} once we check that 
Assumption \ref{ass} of Proposition \ref{prop 2.10} are satisfied 
for $T=\mathbf{T}^\dagger$ and $R=\mathcal{R}$ 
in Assumption \ref{ass}. 
We know that $\mathbf{T}^\dagger$ is free of rank $2$  over $\mathcal{R}$, and is unramified over the set of places $\Sigma$ dividing $Np$;  moreover, $F^+_v(\mathbf{T}^\dagger)$ and $F^-_v(\mathbf{T}^\dagger$ 
are free of rank $1$ over $\mathcal{R}$, so both (1) and (3) are satisfied. 
For (2) we need to check that $H^0(K_\Sigma/K_n,\A^\dagger)$ is a pseudo-null $\Lambda_{\mathcal{R}}$-module. Since $\A^\dagger$ is unramified outside $\Sigma$, the Galois group $\Gal(\bar{\Q}/K_\Sigma)$ acts trivially 
on $\A^\dagger$, so $H^0(K_\Sigma/K_n,\A^\dagger)=H^0(K_n,\A^\dagger)$, which is trivial by Lemma \ref{vanishing}. 
Condition (2) is guaranteed by the fact that $p$ is non-anomalous in 
Assumption \ref{ass2}, after taking $\delta_v=\eta_v^{-1}$ and  
$\theta_v=\Theta^{-1}\varepsilon_\mathrm{cyc}[\varepsilon_\mathrm{cyc}]$, noting that $\theta_v$ factors through the cyclotomic $\Z_p$-extension of $K$. 
\end{proof}

\section{Proofs of the main results} 

The following result proves \cite[Conjecture 9.12]{LV-MM}, 
a definite version of the two-variable Iwasawa main conjecture for Hida families in the anticyclotomic context. 

\begin{theorem}\label{theorem} Suppose Assumptions $\ref{ass1}$ and $\ref{ass2}$ are satisfied, and that the Hida family $\mathbf{f}$ admits a specialisation $f_k$ of weight $k\equiv 2\pmod{p-1}$ and trivial nebentypus. Then the group $\tilde H^1_{f,\mathrm{Iw}}(K_\infty,\mathbf{A}^\dagger)$ 
is a finitely generated cotorsion 
$\Lambda_{\mathcal{R}}$-module and there is an equality 
\[\left(L_p(\mathbf{f}/K)\right)=\mathrm{Char}_{\Lambda_{\mathcal{R}}}\left(H^1_{f,\mathrm{Iw}}(K_\infty,\mathbf{A}^\dagger)^\vee\right)\]
of ideals in $\Lambda_{\mathcal{R}}$. 
\end{theorem} 

\begin{proof} That $\tilde H^1_{f,\mathrm{Iw}}(K_\infty,\mathbf{A}^\dagger)$ is finitely generated follows easily from the topological Nakayama's Lemma. 
The proof of \cite[Theorem 5.3]{CKL} shows the inclusion of the characteristic ideal in the ideal generated by the 
$p$-adic $L$-function (see in particular the last displayed equation in the proof 
of \cite[Theorem 5.3]{CKL}). More precisely, by \cite[Theorem 11.1]{KL} we know that 
$L_p(\mathbf{f}/K)$ is equal, up to units of $\mathbb{I}$, to the self-dual twist of the restriction of 
Skinner--Urban's three-variable $p$-adic $L$-function to the anticyclotomic line (see \cite[\S 4.4]{KL}). 
Combining \cite[Theorem 3.26 ]{SU} and \cite[Lemma 1.2]{Rubin}, we see that the inclusion of 
$\mathrm{Char}_{\Lambda_{\mathcal{R}}}\left(H^1_{f,\mathrm{Iw}}(K_\infty,\mathbf{A}^\dagger)^\vee\right)$
in $\left(L_p(\mathbf{f}/K)\right)$ holds. To get the equality, 
it suffices to establish equality for some classical specialisation, which follows  
in our setting from \cite[Corollary 3]{CKL}. Finally, since $L_p(\mathbf{f}/K)\neq 0$, it follows that 
$H^1_{f,\mathrm{Iw}}(K_\infty,\mathbf{A}^\dagger)$ is $\Lambda_\mathcal{R}$-cotorsion. 
\end{proof}

As a corollary of Theorem \ref{theorem}, we obtain a result in the direction of \cite[Conjecture 9.5]{LV-MM}, 
a definite version of the horizontal non-vanishing conjecture of Howard \cite[Conjecture 3.4.1]{howard-invmath}. Denote $\chi_\mathrm{triv}:\mathcal{R}\pwseries{\Gamma_\infty}\rightarrow\mathcal{R}$ the morphism associate with the trivial character of $\Gamma_\infty$, 
and define
\begin{equation}\label{J}
\mathcal{J}_0=\chi_\mathrm{triv}\left(\Theta_{\infty}(\mathbf{f})\right).\end{equation}

\begin{corollary}\label{theorem2} Let the assumptions be as in Theorem $\ref{theorem}$. If 
$\tilde H^1_f(K,\mathbf{T}^\dagger)$ is a torsion $\mathcal{R}$ module, 
then $\mathcal{J}_0\neq0$.  
\end{corollary} 

\begin{proof}
Since $\tilde H^1_f(K,\mathbf{T}^\dagger)$ is a torsion $\mathcal{R}$-module, it follows from \cite[Corollary 5.5]{LV-Pisa} that 
$\tilde H^1_f(K,V_{f_\kappa}^\dagger)=0$ for 
all but finitely many arithmetic character $\kappa$, where 
$\tilde H^1_f(K,V_{f_\kappa}^\dagger)$ is the extended Bloch--Kato Selmer group of $V_{f_\kappa}^\dagger$. By \cite[Proposition 12.7.13.4(i)]{Nekovar}, this implies that $\tilde H^2_f(K,\mathbf{T}^\dagger)$ is a torsion $\mathcal{R}$-module.
Poitou--Tate global duality \cite[\S0.1]{Nekovar} implies then that $H^1_{f}(K,\mathbf{A}^\dagger)^\vee$ is also a torsion 
$\mathcal{R}$-module. 

Let $I$ be the kernel of $\chi_\mathrm{triv}$.
By Theorem \ref{control theorem}, the kernel and cokernel of the map
\[H^1_{f,\mathrm{Iw}}(K_\infty,\mathbf{A}^\dagger)^\vee/IH^1_{f,\mathrm{Iw}}(K_\infty,\mathbf{A}^\dagger)^\vee\longrightarrow 
H^1_{f}(K,\mathbf{A}^\dagger)^\vee 
\]
are torsion $\I$-modules. 
Since $H^1_{f}(K,\mathbf{A}^\dagger)^\vee$ is a torsion 
$\mathcal{R}$-module, it follows that 
\[H^1_{f,\mathrm{Iw}}(K_\infty,\mathbf{A}^\dagger)^\vee/IH^1_{f,\mathrm{Iw}}(K_\infty,\mathbf{A}^\dagger)^\vee\] is also a torsion $\mathcal{R}$-module, and its characteristic power series is then 
a non-zero element of $\mathcal{R}$. 
By Theorem \ref{theorem}  
we then have $L_p(\mathbf{f}/K)(\chi_\mathrm{triv})\neq 0$.  
The result follows now from Definition \ref{p-adic L-function} 
and \eqref{J}. 
\end{proof}

\bibliographystyle{amsalpha} 
\bibliography{paper}
\end{document}